\documentclass[11pt]{amsart}
\usepackage{amsmath}
\usepackage{amssymb}
\usepackage{amscd}
\usepackage{amsthm}
\usepackage{amsfonts}
\usepackage{subfigure}

\usepackage[all]{xy}
\usepackage{color}

\numberwithin{equation}{section}

\newtheorem{theorem}[equation]{Theorem}

\newtheorem{proposition}[equation]{Proposition}
\newtheorem{prop}[equation]{Proposition}

\newtheorem{lem}[equation]{Lemma}
\newtheorem{lemma}[equation]{Lemma}

\newtheorem{corollary}[equation]{Corollary}
\newtheorem{cor}[equation]{Corollary}

\theoremstyle{remark}

\theoremstyle{definition}
\newtheorem{definition}[equation]{Definition}

\newtheorem{observation}[equation]{Observation}

\newtheorem{example}[equation]{Example}

\newcommand\res{\mathop{\hbox{\vrule height 7pt width .5pt depth 0pt
\vrule height .5pt width 6pt depth 0pt}}\nolimits}

\def\XXint#1#2#3{{\setbox0=\hbox{$#1{#2#3}{\int}$} 
	\vcenter{\hbox{$#2#3$}}\kern-.5\wd0}}

\newcommand{\Hc}{\mathcal{H}}
\newcommand{\Lc}{\mathcal{L}}

\newcommand{\R}{\mathbb R}

\newcommand{\G}{{\mathbb G}}

\def\eps{\epsilon}

\newcommand{\vol}{\operatorname{vol}}
\newcommand{\uno}{{\mathbb 1}}
\usepackage{bbold}

\begin{document}

\title[Sets with constant  normal in the Engel group]{
Regularity of sets with constant horizontal normal in the Engel group}

\author{Costante Bellettini}
\address{Princeton University\\
Fine Hall, Washington Road,
Princeton, NJ 08544-1000, USA and Institute for Advanced Study\\
Einstein Drive, Princeton, NJ 08540 USA}
\email{cbellett@math.princeton.edu}

\author{Enrico Le Donne}
\address{Mathematical Sciences Research Institute\\
17 Gauss Way,
Berkeley, CA 94720-5070,
 USA and Laboratoire de Math\'ematiques\\
B\^atiment 425, Universit\'e Paris Sud 11\\
91405 Orsay\\
France}
\email{ledonne@msri.org}


\renewcommand{\subjclassname}{%
\textup{2010} Mathematics Subject Classification}

\date{January 30, 2012}

\begin{abstract}
In the Engel group with its Carnot group structure we study subsets of locally finite subRiemannian perimeter and possessing constant subRiemannian normal.

We prove the rectifiability of such sets: more precisely we show that, in some specific coordinates, they are upper-graphs of entire Lipschitz functions (with respect to the Euclidean distance).
However we find that, when they are written as intrinsic upper-graphs with respect to the direction of the normal, then the function defining the set might even fail to be continuous. Nevertheless, we can prove that one can always find other horizontal directions for which the set is the upper-graph of a function that is Lipschitz-continuous with respect to the intrinsic distance (and in particular H\"older-continuous for the Euclidean distance).

We further discuss a PDE characterization of the class of all sets with constant horizontal normal.

Finally, we show that our rectifiability argument extends to the case of filiform groups of the first kind.
\end{abstract}

\maketitle

\setcounter{tocdepth}{2}
\tableofcontents

\section{\textbf{Introduction}}

Recent years have witnessed an increasing interest in Geometric Analysis  of
Metric Spaces. A particular role has been played by the class of Carnot groups endowed with   subRiemannian distances.
In this setting, both translations and dilations are present, hence the  theory of differentiation   generalizes. Many notions from Analysis and Geometry have been investigated in subRiemannian Carnot groups.
Function Theory has been  a fruitful study. There have been several fundamental results in the study of maps such as Lipschitz, Sobolev, quasiconformal, and bounded variation.
Another subject of large interest has been Geometric Measure Theory.
Minimal surfaces, sets with finite perimeter, currents, and  rectifiable sets 
have  received particular attention.
As a source of reference, we point out to \cite{Capogna-et-al, Vittone-tesi, LeDonne-lectures},
 and the references therein.

In this paper we intend to contribute to the study of the regularity of some particular class of hypersurfaces in a specific Carnot group, namely the Engel Lie group, which is the lower-dimensional Carnot group of step $3$. 
It is the simplest example for which there is lack of rectifiability results. We consider sets that are the generalization of half-spaces, i.e., they have constant horizontal normal. 
Such sets are doubly important: they appear as tangents of sets with locally-finite horizontal perimeter, 
and   their boundaries  are examples of minimal hyper-surfaces that can be written as entire graphs with respect to the group structure.
For such reasons, it would be fundamental to understand whether such sets are, in some sense, rectifiable.

We shall give an ambivalent study of the regularity of such constant normal sets.
Briefly, we will show that such sets are rectifiable (even in the Euclidean sense) but not necessarily intrinsic Lipschitz graphs in the direction of the normal. 
In fact, there is an example of a set with constant horizontal normal such that, if one writes the set as a upper-graph in the direction of the normal, the function giving the graph is not continuous. These sets are however intrinsic Lipschitz graphs in other horizontal directions. 

Such (horizontal) intrinsic graphs have natural parametrizations and have been extendedly considered in the theory of rectifiable subsets of Carnot groups, as  examples in \cite{Kirchheim-SerraCassano, 
FSSC2006, ASCV, Barone-Adesi-Serra-Cassano-Vittone, 
MSCV,
Bigolin-Vittone,Bigolin-SerraCassano1,Bigolin-SerraCassano2,FSSC2011}.  
It is not clear whether a rectifiability result can be expected to hold for sets with constant horizontal normal in general Carnot groups. 
Nor we have examples of finite-perimeter sets 
for which it fails that, at almost every point of the boundary, the tangent is a half-space.

\subsection{Main setting, terminology, and previous contributions}
Let $\G$ be a Carnot group (see \cite{AKL} for definitions). Let $\mathfrak g$ be the Lie algebra of the left-invariant vector fields in $\G$. By definition, $\mathfrak g$  is stratified. We denote by $V_1$ the first stratum (also known as horizontal layer).

A subset $E$ of a Carnot group is said to have locally finite horizontal perimeter if, for any 
 $X \in V_1$, the distribution $X\uno_E$ is a Radon measure.
Caccioppoli and De Giorgi introduced these sets (in the Euclidean space)  for the study of minimal hyper-surfaces. 
The reason for doing so is the good behavior  of the perimeter, which is the total mass of the vector-valued measure whose components are obtained by differentiating $\uno_E$ in the  directions of a fixed basis of $V_1$.
In fact, the perimeter is lower semicontinuous and induces a locally compact topology on the class of finite-perimeter sets. Hence, it becomes easy to show existence of minimal surfaces.

A set $E$ in a Carnot group $\G$ is said to have constant horizontal normal if there exists a horizontal left-invariant vector field $X$ on $\G$ and there exists a decomposition $\R X \oplus V_1^\dagger$  of $V_1$ 
 with the following property:
\begin{itemize}
\item 
the distributional derivative $X\uno_E$ of the characteristic function $\uno_E$ of $E$ in the direction of $X$ is a positive Radon measure;
\item for all $Y\in V_1^\dagger$, the distribution    $Y\uno_E$ vanishes.
\end{itemize}

One should notice that the space $V_1^\dagger$ is uniquely defined by $E$, unlike the vector $X$. However, if we fix a scalar product on $V_1$ and require that $X$ is a unit vector orthogonal to $V_1^\dagger$, then $X$ is unique and it is called {\em the} normal of $E$.

In \cite{fssc}, the three authors extended a result of De Giorgi by proving rectifiability of sets with locally finite horizontal perimeter in Carnot groups of step $2$. Following De Giorgi's strategy, they obtained this result, by showing  that almost every tangent is a set of constant normal and that constant-normal sets  are in fact half-spaces.
Alas, they noticed that this latter fact does not hold in higher-step Carnot groups.
In fact, in
\cite[Example 3.2]{fssc},
they gave the first explicit example of  a subset of the 
Engel group with constant normal that is not a half-space. More examples have then been given in  \cite{AKL} and we will be adding some more in the present work.

In  \cite{Barone-Adesi-Serra-Cassano-Vittone}, the three authors showed that sets with constant horizontal normal are \textit{calibrated sets}: the calibration that they used is the scalar product with the normal. This calibration method implies that, in any Carnot group, boundaries of sets with constant horizontal normal are minimal surfaces, just as it happens in the Euclidean framework. 
Guided by the Euclidean experience, one could expect fairly good results on the smoothness of calibrated sets: the classical regularity theory for minimal sets in $\R^n$ indeed argues, in its key steps, as follows. First write the set, locally around a point where we have a tangent plane, as a graph on the tangent plane; then prove that the normal is H\"older continuous. It is now crucial the fact that, in this Euclidean setting, one can further improve regularity to $C^1$ (for the original proofs see \cite{DeGiorgi61} and \cite{Giusti-book}).

Going back to the subRiemannian framework, one can write any  constant-normal set as an intrinsic upper-graph in the direction of the normal.
Such graphs have been considered in 
\cite{Kirchheim-SerraCassano, 
FSSC2006, ASCV, Barone-Adesi-Serra-Cassano-Vittone, 
MSCV,
Bigolin-Vittone,Bigolin-SerraCassano1,Bigolin-SerraCassano2,FSSC2011} and give  canonical  parametrizations.
By the work of \cite{Monti-Vittone}, 
it was expected that, if the uppergraphs of a function were giving sets with controlled normal (e.g., constant),
then the function had some regularity, as it happens in the classical Euclidean case that we mentioned before.

We shall give examples in the Engel group of sets with constant horizontal normal where the function is not even continuous. Here the choice of a specific normal direction (namely, of a scalar product on $V_1$) will be crucial. On the other hand we can prove that, by choosing other horizontal directions, we can express the set as upper-graph of a function taking values in the new direction, and this function is intrinsically Lipschitz continuous, in particular H\"older continuous.

\subsection{Overview of results}
We recall now the definition of the Carnot group of interest to us, the Engel group. The {\em Engel algebra} is the Lie algebra generated, as vector space,  by four vectors $X_1$, $X_2$, $X_3$, $X_4$, with relations
\begin{eqnarray}\label{engel5}
&[X_1 , X_2 ] = X_3 \qquad\text{and}\qquad [X_1 , X_3 ] = X_4,& \\
&[X_1 , X_4 ] =[X_2 , X_4 ] =[X_2 , X_3 ] =[X_3 , X_4 ] =0.& \nonumber 
\end{eqnarray}
Such an algebra is nilpotent of step $3$ and stratified by the strata
$$V_1:=\R X_1\oplus \R X_2, \; \; \; \; V_2:= \R X_3,\; \;  \; \; V_3:= \R X_4.$$
The Engel group is defined as the unique connected and simply connected   Lie group with the Engel algebra as Lie algebra.
Through the paper we denote by $\G$ such a group.

We endow the Engel group $\G$ with some Haar measure $\vol_\G$.
Given a measurable set $E\subset \G$ and a left-invariant vector field $X\in Lie(G)$, we write
$$\quad X \uno_E\geq 0 \qquad\qquad \text{ ( resp.  }X \uno_E= 0\,\text{)} $$
if, for all $\phi\in C_c^\infty(\G)$ with $\phi\geq 0$,
$$\quad -\int_E X\phi \;d\vol_\G\geq 0 \qquad\qquad \text{ ( resp.  }\int_E X\phi \;d\vol_\G= 0\,\text{)}. $$

Since the flow of a  left-invariant vector field   is a right translation, then   
 the flow of such a vector field preserve the Haar measure, which on a nilpotent group is always biinvariant.
In other words, any element of the Lie algebra is a divergence-free vector field on the manifold $\G$, endowed with a Haar measure   $\vol_\G$.

\begin{definition}[Constant horizontal normal]
Let $V_1$ be the first stratum of $Lie(\G)$. Fix a scalar product $\langle\cdot|\cdot\rangle$ on $V_1$.
A set $E\subseteq \G$ is said to have constant horizontal normal $X\in Lie(G)$ if
$X\in V_1$, $X \uno_E\geq 0$, and
$$Y\in V_1, \langle X|Y\rangle=0\implies Y \uno_E= 0.$$
\end{definition}

Regarding the next definition, recall that, being  connected, simply connected, and nilpotent, the Engel group $\G$ is diffeomorphic to $Lie(G)$, via the exponential map, and so is diffeomorphic to $\R^4$ .

\begin{definition}[Lipschitz domain]
A set $E\subseteq \G$ is called a Lipschitz Euclidean domain if, for one (and thus for all) diffeomorphisms $f:\G \to \R^4$, the set $f(E)$ is a  Lipschitz domain of $\R^4$. Namely,  $f(E)$ is an open set and any point on the boundary has a neighborhood in which the set can be described as the upper-graph of a Lipschitz map of three variables.
\end{definition}

Our first result is the following.
\begin{theorem}\label{teorema1}
If $E$ is a subset of the Engel group $\G$ that has constant horizontal normal, 
then there exists a Lipschitz domain $\tilde E\subset \G$ that is equivalent to $E$, i.e., it is such that $\vol(E\Delta\tilde E)=0$.
\end{theorem}

The strategy of the proof is as follows. Let $Y_1=X$ be the normal of $E$. Take $Y_2\in V_1$ with $\langle X|Y_2\rangle=0$.
Hence $Y_1 \uno_E\geq 0 $ and $Y_2 \uno_E= 0 $. 
Using a result of \cite[Proposition 4.7]{AKL}, we get $Y_3,Y_4\in Lie(\G)$ such that $Y_1,Y_2,Y_3,Y_4$ form a basis, $Y_3 \uno_E\geq 0 $, and  $Y_4 \uno_E\geq 0.$
Take the Lebesgue representative $\tilde E$ of $E$.
The set $\tilde E$ will have the property that, for all $p\in \partial \tilde E$ and for all $Z=\sum_{j=1}^4 \alpha_jY_j$ with $\alpha_j>0$, one has that 
$$\{t\in \R\;:\; p\exp(tZ)\in \tilde E\}=(0,+\infty).$$
In other words, the `cone'
$$C_p:=\left\{p\exp\left(\sum_{j=1}^4 \alpha_jY_j\right)\;:\;\alpha_j>0\right\}$$
does not intersect $\tilde E$.
Finally, a standard cone criterion gives the Lipschitz regularity of $\partial \tilde E$.

In the proof we get a cone $C_p$ that in exponential coordinates of second kind is constant, i.e., $C_p$ is a left translation of $C_q$, for $p,q\in \partial \tilde E$.
Thus we might conclude that $\partial \tilde E$ is an entire Lipschitz graph.

\begin{theorem}
\label{cor:37} Let $\G$ be the Engel group. Let $\Psi:\R^4 \to \G$ be the exponential coordinates of the second kind. There exists $L \geq  0$ such that, for all horizontal vector $X$, there exists a basis $w_1,\ldots,w_4\in \R^4$ with the following property. If $E\subset\G$ is any subset that has constant horizontal normal $X$, 
then there exists an $L$-Lipschitz map $h:\R^3\to \R$ such that $E$ is equivalent to 
$$\{ \Psi(a_1 w_1 +\ldots +  a_4 w_4)\;:\; a_1,a_2,a_3\in \R, a_4>h(a_1,a_2,a_3)\}.$$
\end{theorem}

Our rectifiability argument can be extended to all those filiform groups that are generated by two vectors $X$ and $Y$ and for which $Y$ and all the iterated commutators of $X$ and $Y$ are in the center. These are the so-called ``filiform groups of the first kind''. See Section \ref{filif-sec} for a discussion on filiform group and the proof of the result.

\medskip

Going back to the Engel group, we now turn our attention to the expression of $E$ as an \textit{``algebraically intrinsic horizontal graph''} (see Section \ref{intrinsic}).  This means that we express $E$ as 
union of half-flow-lines in a horizontal direction. 
 More precisely, let $X$ be a horizontal direction. 
Let $W$ be a  subgroup  that is complementary to  $\exp(\R X )$.
Let $T: W\to \R$ be any function.
We say that $E$ is an {\em intrinsic horizontal upper-graph} if
$$E=\{ w \exp(t X ) : w\in W, t > T(w)\}.$$
One should observe that the point $w \exp(t X ) $ is the  flow from $w$ for time $t$ in the direction $X$. Understanding the regularity of $T$ is then our next task.

The most natural horizontal direction to use is the direction of the normal. Nevertheless we find that

\begin{description}
 \item[(i)] (see \textbf{Theorem \ref{thm:noncont}}) There are examples of subsets $E\subset \G$ with constant horizontal normal $X_2$ and with the property that, when they are  written as upper-graphs in the  direction of $X_2$, the function 
for which they are upper graphs is  not continuous.  

 \item[(ii)] (see \textbf{Theorem \ref{theorem:intrinsiclipshitz}}) Let $E\subset \G$ be an arbitrary set having constant horizontal normal $X$. Whenever $E$ is written as intrinsic horizontal upper-graph using a horizontal direction $Y$ with $\langle X, Y \rangle \neq 0$ and $X$ not parallel to $Y$, then the function for which it is upper-graph is Lipschitz continuous with respect to the intrinsic Carnot-Caratheodory distance on $\G$. In particular the function is H\"older continuous for the Euclidean distance.

\end{description} 

We note that the `Euclidean Lipschitz continuity'
obtained in Theorem \ref{cor:37} requires the expression of $E$ as a graph with respect to a non-horizontal direction, see Theorem \ref{theorem:notlipshitz}. This confirms that the natural notion (intruduced by Franchi, Serapioni, and Serra Cassano in \cite{FSSC2006}, see also \cite{FSSC2011}) of
Lipschitz continuity to be used for intrinsic horizontal graphs in the 
subRiemannian context is the one with respect to the
intrinsic distance.


\medskip

To complete the description of sets with constant horizontal normal in the Engel group, in exponential coordinates of second kind we are able to give a characterization of any such set as upper-graph of a function that satisfies a partial differential inequality. Roughly speaking we can prove that a set has finite perimeter and constant horizontal normal $X_2$ if and only if it is of the form  $\{x_2 > G(x_3, x_4)\}$ for a BV function $G:\R^2 \to \R$ that satisfies the following partial differential inequality:
for all $h\in C^\infty_c(\G)$ such that $h\geq 0$,  it holds
\begin{equation}\label{PDI-intro}
(\langle \partial_3G, h\rangle)^2+2\langle\partial_4 G, h\rangle \langle {\Lc}^2, h \rangle \leq 0 .
\end{equation}
Here  ${\Lc}^2$ denotes the Lebesgue measure on $\R^2$ and  $\langle \,,\,\rangle$ denotes the pairing of distributions and smooth test functions.
The precise statement of the result (see Proposition \ref{prop:characterize}) requires to extend the target of $G$ to $\{-\infty\} \cup \R \cup \{+\infty\}$.
Minimal graphs of functions that also assume the values $+\infty$ and $-\infty$ have already appeared in the  Euclidean setting.  Namely, Mario Miranda considered them in the solution of the Dirichlet's problem for the minimal surfaces equation, see \cite{Miranda77} and \cite[Chapter 16]{Giusti-book}.

\subsection{Acknowledgements}
Both authors would like to thank ETH Z\"urich  for its supporting
research environment  
while part of this work was conducted. 
This paper has benefited  from numerous discussions with Luigi Ambrosio, Bruce Kleiner, Raul Serapioni, Francesco Serra Cassano, and Davide Vittone.
Special thanks go to them.


\section{\textbf{Getting more monotone directions}}
Let $E$ be a subset of the Engel group $\G$ that has constant horizontal normal $X\in Lie(\G)$.
 Let $Y$ be a vector in $V_1$ that is orthogonal to $X$.
Notice that the line $\R Y$ is independent from the scalar product chosen on $V_1$.
Now we face a dichotomy: either $Y$ is parallel to the vector $X_2$ of the definition of the Lie algebra representation \eqref{engel5} of $Lie(\G)$, or not. In the second case, we  show that we can change the basis of $Lie(\G)$ and assume that $Y=X_1$.

The case $Y=X_2$ is easy to handle and in fact we show that $E$ is (equivalent to) a half-space.
The case 
$Y=X_1$
is more complex and in this case it is not true that $E$ is a half-space, as it was previously shown in \cite{fssc}.

\subsection{Easy case: sets with normal   $X_1$}
Let $\G$ be the Engel group whose Lie algebra is generated by $X_1$, $X_2$, $X_3$, $X_4$ with relations
\eqref{engel5}.
\begin{lemma}\label{lem1}
Let $E\subset \G$. Assume that $X_1 \uno_E\geq 0$ and $X_2 \uno_E= 0$. Then $E$ is a vertical half-space.
\end{lemma}
For the definition and other characterization of  half-spaces see \cite{AKL}.
We make use of the following property of stability of monotone directions.
\begin{proposition}[{\cite[Proposition 4.7]{AKL}}]\label{AKL prop}
Let $X, Y\in Lie(\G)$ and $E\subset \G$. Assume that $X \uno_E= 0$ and $Y \uno_E\geq 0$. Then 
$({\rm Ad}_{\exp(X)}Y )\uno_E\geq 0$.
\end{proposition}
Recall that 
$${\rm Ad}_{\exp(X)}Y = e^{{\rm ad}_X}Y=Y+[X,Y]+\dfrac{1}{2}[X,[X,Y]],$$
in a $3$-step group.

\proof[Proof of Lemma \ref{lem1}] 
Applying Proposition \ref{AKL prop} with $X=t X_2$, $Y=X_1$, and $t\in \R$, we get that the vector field 
$$Z:={\rm Ad}_{\exp(tX_2)}X_1 =X_1+[tX_2,X_1]+\dfrac{1}{2}[tX_2,[tX_2,X_1]]=X_1-t X_3 $$
is such that $Z\uno_E\geq 0$, for all $t\in \R$. Letting $t\to +\infty$ and $t\to -\infty$, respectively, we get that both
$-X_3 \uno_E\geq 0$ and $X_3 \uno_E\geq 0$. Hence 
$X_3 \uno_E= 0$.
Apply again the proposition with $X=t X_3$, $Y=X_1$, and $t\in \R$.
Thus the vector
$$Z':={\rm Ad}_{\exp(tX_3)}X_1 =X_1+[tX_3,X_1] =X_1-t X_4 $$
is such that $Z'\uno_E\geq 0$, for all $t\in \R$.
Arguing as before, we conclude that $X_4 \uno_E= 0$.
By the BV characterization of vertical half- spaces, see {\cite[Proposition 4.4]{AKL}}, we are done.
\qed

\subsection{Hard case: sets with normal  $X_2$}

We first argue that if the normal is not $X_1$ then we can assume that the normal is $X_2$.
\begin{lemma}
Let $\mathfrak g$ be the Engel algebra with basis $X_1$, $X_2$, $X_3$, $X_4$ and relations
\eqref{engel5}. Let $X=\alpha X_1 +\beta X_2$ with $\alpha, \beta\in \R$ and $\alpha\neq 0$.
Then there exists a Lie algebra endomorphism $\psi$ of $\mathfrak g$ such that 
$\psi X_1=X.$
\end{lemma}
\proof
Define $\psi:\mathfrak g\to \mathfrak g$ by the property
$$\psi X_1=\alpha X_1 +\beta X_2
\qquad \text{ and }\qquad \psi X_2= X_2,$$
and
$$\psi X_3=\alpha X_3\qquad \text{ and }\qquad \psi X_4= \alpha^2 X_4.$$
It is straightforward
\footnote{Here is the calculation: 
$\psi[X_1,X_2]=\psi X_3=\alpha X_3= \alpha [X_1,X_2]=[\psi X_1,\psi X_2]$ and
$\psi[X_1,X_3]=\psi X_4= \alpha^2 X_4= \alpha^2 [X_1,X_3]=[\psi X_1,\psi X_3]$.
}
 to check that such a $\psi$ is an isomorphism. \qed
 
By the above lemma, the following fact is immediate.
\begin{corollary}\label{dicotomia}
Let $E$ be a subset of the Engel group $\G$. Let  $X, Y\in V_1\subset Lie(\G)$ linearly independent.
Assume $X \uno_E\geq 0$ and $  Y \uno_E= 0.$
Then there exists a basis $X_1$, $X_2$, $X_3$, $X_4$ of $Lie(\G)$ with relations
\eqref{engel5} such that 
\begin{itemize}
\item either $X_1 \uno_E\geq 0$ and $  X_2 \uno_E= 0,$
\item or $X_1 \uno_E= 0$ and $  X_2 \uno_E\geq 0.$
\end{itemize}
\end{corollary}
In other words, we only need to study the cases where either $E$ has normal $X_1$ or it has normal $X_2$.
Since we already solved the first case, let us focus now on the second.

If one applies  Proposition \ref{AKL prop} to the case of constant equal to $X_2$, one obtains the following.
\begin{lemma}\label{lem2}
Let $E\subset \G$. Assume that $X_2 \uno_E\geq 0$ and $X_1 \uno_E= 0$. Then, for all $t\in \R$, the vector
$$Z_t:= X_2+t X_3 +\dfrac{t^2}{2} X_4$$
is such that
$Z_t\uno_E\geq 0$.
In particular,
$$X_4\uno_E\geq 0\qquad\text{ and }\qquad (X_2+2X_3 +2 X_4)\uno_E\geq 0.$$
\end{lemma}

One should notice at this point that the four vectors
$$Y_1:=X_1, \quad Y_2:=X_2, \quad Y_3:=X_4, \quad Y_4:=X_2+2X_3 +2 X_4$$
forms a  basis of $Lie(\G)$. Moreover,  for all $j=1,\ldots, 4$,
we just proved that $Y_j \uno_E\geq 0$.
The next proposition will then conclude the proof of Theorem \ref{teorema1}.
\begin{proposition}\label{prop1}
Let $\G$ be any Carnot group. Let $E\subset \G$. Let $Y_1,\ldots, Y_n$ be a basis of $Lie(\G)$.
Assume that $Y_j \uno_E\geq 0$, for all $j=1,\ldots, n$.
 Then $E$ is equivalent to an (Euclidean) Lipschitz domain.
\end{proposition}

We need to postpone the proof of Proposition \ref{prop1}, since we need some preliminaries for it. Namely, we need to choose a good representative for the set $E$.

\proof[Proof of Theorem \ref{teorema1}]
According to Corollary \ref{dicotomia}, we might consider two cases: either the set $E$ has normal $X_1$ or it has normal $X_2$.
In the first case, Lemma \ref{lem1} implies that   $E$ is a vertical half-space, and so a Lipschitz domain.

If the normal is $X_2$, 
Lemma \ref{lem2} give us four linearly independent monotone directions, i.e., $Y_j \uno_E\geq 0$.
Hence, Proposition \ref{prop1} concludes.
\qed

\section{\textbf{Regularity of sets with normal $X_2$}}

In Section \ref{lipdom} we will show Theorem \ref{teorema1} by proving that the Lebesgue representative $\tilde{E}$ of a set $E$ having constant horizontal normal $X_2$ satisfies a ``cone criterion'' and is therefore an Euclidean Lipschitz domain.

We will then, in subsection \ref{inamodel}, fix a model of the Engel group and show that in this model, by suitably rotating coordinates, $\tilde{E}$ can actually be expressed as the subgraph of an entire Lipshitz function. In the last subsection \ref{classofex} we characterize all sets having constant horizontal normal $X_2$.

\subsection{Sets with normal $X_2$ are Lipschitz domains}
\label{lipdom}
We need to choose a good representative for the set $E$. In fact we want to have an equivalent set $\tilde E$ for which {\em all} line flows of $Y_j$, $j=1,\ldots, n$, meet $\tilde E$ in a half-line. Such a fact will also be useful for writing $\partial\tilde E$ as a  graph. The fact that such graphs will be uniformly Lipschitz is because, by assumption, they avoid  cones that are left translations of the same fixed cone.

The strategy requires the study of the Lebesgue representative of our original set, which allows us to obtain monotonicity along every flow line on every direction $Y_j$. Subsequently, we show that the topological boundary of this new set is  locally a Lipschitz graph. 

\medskip

Recall that if $X$ is a left-invariant vector field in a Lie group $\G$, i.e., $X\in Lie(\G)$, then 
its flow is a right translation. Namely, 
 $$\Phi_X(p,t)=p\exp(tX), \qquad \forall p\in \G.$$

\begin{lemma}\label{cortdp'}
Let $\G$ be any Carnot group. Let $E\subset \G$. Let $X\in Lie(\G)$.
Assume that $X \uno_E\geq 0$.
Then, for any $t>0$, we have that (almost everywhere) it holds
$  \uno_E\leq \uno_E \circ\Phi_X(\cdot,t)$. In particular, it is true that a.e.
$$  \uno_{E\exp(X)}\leq  \uno_E.$$
\end{lemma}

\begin{proof} 
Being $X$ a divergence free vector field on the manifold $\G$, endowed with a Haar measure   $\vol_\G$, we can prove that: if $u\in L^1_{\rm loc}(\G)$
satisfies $Xu\geq0$ in the sense of distributions, then, for all
$t>0$, $  u\circ\Phi_X(\cdot,t)\geq u$ $\vol_\G$-a.e. in $\G$. The statement of the lemma then follows immediately.

What we need to show is that, for any non-negative $g\in C^1_c(\G)$, the map
$t\mapsto\int_\G g u\circ\Phi_X(\cdot,t)\,d\vol_\G$ is increasing in
$t$. Indeed, the semigroup property of the flow, and the fact
that $X$ is divergence-free yield
\begin{eqnarray*}
&&\int_\G g \, u\circ\Phi_X(\cdot,t+s)\,d\vol_\G- \int_\G g \,
u\circ\Phi_X(\cdot,t)\,d\vol_\G \\&=& \int_\G u \,
g\circ\Phi_X(\cdot,-t-s)\,d\vol_\G- \int_\G u \,
g\circ\Phi_X(\cdot,-t)\,d\vol_\G \\&=&\int_\G u \,
g\circ\Phi_X(\Phi_X(\cdot,-s),-t)\,d\vol_\G-\int_\G u \,
g\circ\Phi_X(\cdot,-t)\,d\vol_\G\\&=& -s\int_\G u\,
X(g\circ\Phi_X(\cdot,-t))\,d\vol_\G+o(s)\\&=& s\int_\G 
(g\circ\Phi_X(\cdot,-t)) \,Xu\,d\vol_\G+o(s),
\end{eqnarray*}
which, recalling that $Xu \geq 0$, yields that $t\mapsto\int_\G g \, u\circ\Phi_X(\cdot,t)\,d\vol_\G$ is (weakly) increasing in
$t$.
\end{proof}

\begin{lemma}\label{lem5}
Let $\G$ be any Carnot group. Let $E\subset \G$. Let $Y_1,\ldots, Y_k\in Lie(\G)$.
Assume that $Y_j \uno_E\geq 0$, for all $j=1,\ldots, k$.
Then there exists $\tilde E$ such that $\vol_{\G}(E\Delta\tilde E)=0$ and, for all $p\in \G$ and  $j=1,\ldots, k$,
there exists $T \in [-\infty, +\infty]$ such that
$$\{t\in \R\;:\; p\exp(tY_j)\in \tilde E\}=(T,+\infty).$$
\end{lemma}

\proof
In a Carnot group, such as $\G$, Haar measures are both left- and right-invariant. In this proof we will make use of the fact that $\vol_{\G}$ is right-invariant. Flows of left-invariant vector fields are right translations, thus isometries for such any right-invariant distance. The balls $B_r$ considered in this proof are to be understood with respect to a fixed right-invariant metric. 

Let $\tilde E$ be the Lebesgue representative of $E$, i.e. the set of points having density $1$:
$$x \in \tilde E \Leftrightarrow \lim_{r \to 0} \frac{\vol_\G(B_r(x) \cap E)}{\vol_\G(B_r(x))} =1.$$

By the Lebesgue-Besicovitch Differentiation Theorem, $\tilde E$ and $E$ agree $\vol_\G$-a.e. We claim that $\tilde E$ fulfils the requirements of Lemma \ref{lem5}. Indeed, what we need to prove is: let $p \in \tilde E$, $X$ a left-invariant vector field such that $X \uno_E\geq 0$: then for any $t>0$ the point $y=p\exp(tX)$ belongs to $\tilde E$.

The vector field $X$ is smooth and the flow $\Phi_X(\cdot,t)$ is an isometry for the right-invariant metric, so it sends balls to balls of the same size. In the following denote by $y$ the point $\Phi_X(\cdot,t)(p)$. 
We assume $p \in \tilde E$, so $\displaystyle \lim_{r \to 0} \frac{\vol_\G(B_r(p) \cap E)}{\vol_\G(B_r(p))}=1$. 
By the invariance of $\vol_\G$ along the flow we have $\vol_\G(B_r(p) \cap E)=\vol_\G(B_r(y) \cap E\exp(X))$ and $\vol_\G(B_r(p)) = \vol_\G(B_r(y))$. With the aid of Lemma \ref{cortdp'}, we then have $\vol_\G(B_r(p) \cap E) \leq \vol_\G(B_r(y) \cap E)$. Altogether we can write $$\frac{\vol_\G(B_r(y) \cap E)}{\vol_\G(B_r(y))} \geq \frac{\vol_\G(B_r(p) \cap E)}{\vol_\G(B_r(p))}.$$
Therefore we have 
$$\lim_{r \to 0} \frac{\vol_\G(B_r(p) \cap E)}{\vol_\G(B_r(p))}=1
\quad
 \implies
\quad
\lim_{r \to 0} \frac{\vol_\G(B_r(y) \cap E)}{\vol_\G(B_r(y))}=1,$$   and the lemma is proved.
\qed

\begin{observation}\label{osservazione}
If $Y_1, ... , Y_k$ form a basis of $Lie(\G)$, then the set $\tilde E$ is actually open. Indeed, let $q$ be a point on the topological boundary of $E$ and let us show that the upper density of $E$ at $x$ is striclty less than one. Any direction in the convex envelope of some given monotone directions is in turn monotone, thus the whole cone $\hat{Y}$ obtained as convex envelope of $Y_1, ... , Y_k$ is made of monotone directions. Under the assumption that $Y_1, ... , Y_k$ form a basis of $Lie(\G)$, this cone has non-empty interior. The complement of $E$ is also a set with constant horizontal normal and contains a sequence of points $q_n$ converging to $q$. By means of lemma \ref{lem5}, the cone $q_n-\hat{Y}$ is all contained in the complement of $E$. The interior of $\hat{Y}$ is non-empty, so as $q_n \to q$ the upper density of $E$ at $q$ must be stricly less than $1$.

Hence, by Lemma \ref{lem2}, every set in the Engel group that has normal $X_2$, has a representative that is open and is satisfies the conclusion of Lemma  \ref{lem5}.
\end{observation}

%
%

\proof[Proof of Proposition \ref{prop1}]
By the lemma just proved, we can assume the following: let $E\subset \G$, where $\G$ is a Lie group, let $Y_1,\ldots, Y_n$ be a basis of $Lie(\G)$ such that for all $p\in \partial E$ and  $j=1,\ldots, k$,
we have that
$$\{t\in \R\;:\; p\exp(tY_j)\in \tilde E\}=(0,+\infty).$$
We want to show that $\partial E$ is locally a Lipschitz graph.

\medskip

Fix $p_0\in\partial E$. For all $p\in\G$ consider the set 
$$C_p:=\left\{p\exp\left(\sum_{j=1}^n \alpha_jY_j\right)\;:\;\alpha_j>0\right\}.$$
Fix a (smooth) coordinate chart $\varphi:U\to \R^n$ from a compact neighborhood of $p_0$.
Since $\varphi$ is smooth and $C_p$ change smoothly in $p$, then, for all $p\in U$, the set $\varphi(C_p)$ changes smoothly. Thus one can find a fixed set $C\subseteq \R^n$ of the form
$$C :=\left\{ \sum_{j=1}^n \alpha_jv_j\;:\;\alpha_j>0\right\},$$
for some basis $v_1,\ldots, v_n$ of $\R^n$, such that
$$ \varphi(\partial E)\cap(x+C)=\emptyset,\qquad \forall x\in \varphi(\partial E).$$
Note that consequently we also have that $ \varphi(\partial E)\cap(x-C)=\emptyset.$
By a standard argument, e.g., see \cite[Theorem 2.61, page 82]{Ambook}, one can write $ \varphi(\partial E)$ as a graph in any direction $v\in C$ with respect to any hyperplane $\Pi$ such that 
$\Pi\cap(C\cup-C)=\emptyset$.
\qed

\begin{observation}
As a byproduct we get of course that the set $E$ has \textit{rectifiable} boundary. 
\end{observation}

\subsection{Further regularity in a model of the Engel group}
\label{inamodel}

We will fix, for this subsection and the next, the following model of the Engel group and work in it, proving that a set with constant horizontal normal is, in suitable coordinates, the subgraph of an entire Lipschitz function.

On $\R^4$ with coordinates $x_1, x_2, x_3, x_4$, we consider the following vector fields:
\begin{eqnarray} \label{VFE}
&&X_1 = \partial_1, \nonumber\\ 
&&X_2 = \partial_2 +x_1 \partial_3 + \dfrac{x^2_1}{2} \partial_4 ,\\\nonumber
&&X_3 = \partial_3 + x_1 \partial_4 ,\\\nonumber
&&X_4 =  \partial_4 .
\end{eqnarray}
Such vector fields form a Lie algebra which is $4$-dimensional. Their      only non-trivial 
brackets are 
\begin{equation} \label{engel5 vecchia}
[X_1 , X_2 ] = X_3 ,\qquad [X_1 , X_3 ] = X_4 .
\end{equation}
Therefore, such an algebra is isomorphic to the {\em Engel Lie algebra}.
Using the general theory of (nilpotent) Lie groups one can prove that there exists a (unique) product on $\R^4$ for which the vector fields $X_1, X_2, X_3 , X_4$ are left-invariant (and therefore a basis of the Lie algebra).

Actually, the coordinates for the Engel group that we are using are called the exponential coordinates of the second kind. Namely, if $X_1, X_2, X_3 , X_4$ are a basis of the Lie algebra that satisfies \eqref{engel5 vecchia}, then the map
$$(x_1,x_2,x_3,x_4)\mapsto \exp(x_4X_4)\exp(x_3X_3)\exp(x_2X_2)\exp(x_1X_1)$$
is a diffeomorphism between $\R^4$ and the Engel group. Moreover, the vectors $X_1, X_2, X_3 , X_4$ are pulled back to $\R^4$ to the vector fields as defined in \eqref{VFE}.

Recall that in a Lie group $G$ there is a differential geometric interpretation for the product between an element $p\in G$ with the image $\exp(tX)$ of a multiple of a left-invariant vector field $X$ via the exponential map.
Indeed, one has the formula
\begin{equation}\label{exponential flow}
p\cdot\exp(tX)=\Phi_X^t(p),
\end{equation}
where $\Phi_X^t(p)$ denotes the flow of $X$ after   time $t$ starting from $p$.

The aim of the section is to study those sets that are invariant in the direction of $X_1$ and are monotone in the direction of  $X_2$.
Namely, let $E\subseteq \R^4$ be an open set (we always take the Lebesgue representative), we say that $E$ is {\em $X_2$-calibrated}  if the following two properties holds:
\begin{description}
\item[$X_1$-invariance] if $p\in E$ then, for any $t\in\R$,  $p\exp(tX_1)\in E$;
\item[$X_2$-monotonicity]  for all $p\in\R^4$, the set $\{t\in\R:p\exp(tX_2)\in E\}$ is an open half-line of the form $(T,+\infty) $ for some $T\in\{-\infty\}\cup\R\cup\{+\infty\}$.
\end{description}

Therefore, if $E$ is an $X_2$-calibrated set then $E$ has constant normal equal to $X_2$, i.e., $X_2 \uno_E\geq 0$ and $X_1 \uno_E= 0$. Viceversa,
by Lemma \ref{lem5} and Observation \ref{osservazione}, any set $E$ with   normal   $X_2$ admits an  $X_2$-calibrated set $\tilde E$ that is equivalent to $E$.

\medskip

By formula \eqref{exponential flow}, we can calculate a product $p\cdot\exp(tX)$ without knowing an explicit formula for the product. 
Let us consider the two cases
when $X$ is $X_1$ or $X_2$ as above.

Regarding the flow of $X_1$, we need to solve the ODE
\begin{equation}\label{ODE flow 1}
 \left\{ \begin{array}{ccl}
 \gamma(0)&=&p \\ \\
\dot\gamma(t) &=& (X_1)_{\gamma(t)}.
\end{array} \right.
\end{equation}
Writing $\gamma=(\gamma_1,\gamma_2,\gamma_3,\gamma_4)$ and using the definition of $X_1$, the second inequality becomes
$(\dot\gamma_1(t),\dot\gamma_2(t),\dot\gamma_3(t),\dot\gamma_4(t))=\partial_1=(1,0,0,0).$
Integrating, we have
$$\gamma_1(t)=p_1+t,\; \gamma_2(t)=p_2,\; \gamma_3(t)=p_3,\; \gamma_4(t)=p_4.$$
Thus,
$$p\cdot\exp(tX_1)= (p_1+t, p_2, p_3 ,p_4). $$

Regarding the flow of $X_2$, we consider the ODE
\begin{equation}\label{ODE flow 2}
 \left\{ \begin{array}{ccl}
 \gamma(0)&=&p \\ \\
\dot\gamma(t) &=& (X_2)_{\gamma(t)}=(0,1, \gamma_1(t),(\gamma_1(t))^2 / 2).
\end{array} \right.
\end{equation}
Integrating, we have
$$\gamma_1(t)=p_1,\; \gamma_2(t)=p_2+t,\; \gamma_3(t)=p_3+p_1 t,\; \gamma_4(t)=p_4+p_1^2 t/2.$$
Thus,
\begin{equation}\label{Flow X2}
p\cdot\exp(tX_2)= (p_1, p_2+t, p_3 +p_1 t,p_4+p_1^2 t/2). 
\end{equation}

Thus we replace the previous definition:
\begin{definition}[$X_2$-calibration]
An open  set $E\subseteq \R^4$ is called {\em $X_2$-calibrated} if
\begin{description}
\item[i) ]  if $p\in E$ then, for any $t\in\R$,  $p+(t,0,0,0)\in E$;
\item[ii) ]  for all $p\in\R^4$ there exists $T\in\{-\infty\}\cup\R\cup\{+\infty\}$ such that
$$\{t\in\R:p +(0,t,p_1 t,p_1^2t/2)\in E\}=(T,+\infty).$$
\end{description}
\end{definition}
Since the set $E$ is assumed to be open, condition ii) is equivalent to the following  condition:
\begin{description}
\item[ii') ] 
$$ p\in E, t>0\implies p_t:=p+(0,t,p_1 t,p_1^2t/2)\in E $$
\end{description}

\medskip

\begin{example}\label{example}
Let $g:\R\to\R$ be a non-increasing and upper semi-continuous function.
Consider the set $$E:=\{x\in\R^4 : x_2 > g(x_4)\}.$$
Since $g$ is assumed upper semi-continuous, then $E$ is an open set. Then we claim that the set $E$ is $X_2$-calibrated.
Indeed, such a fact can be seen as a consequence of Proposition \ref{prop:characterize} from next section in which we give a characterization of sets with constant normal. However, we present here a direct and detailed proof of such a claim. Property i) is obvious, for such an $E$, since in the definition of $E$ the variable $x_1$ does not appear. 

Let us show property ii'). If $p\in E$, then $ p_2 > g(p_4)$. Now, if $t>0$, we have that $p_2+t>p_2$ and $g(p_4+p_1^2t)\leq g(p_4)$, being $g$ 
non-increasing.
Thus, $p_2+t -g(p_4+p_1^2t)\geq p_2 - g(p_4)> 0$ and so $p+(0,t,p_1t,p_1^2t)\in E$.
QED
\end{example}

Now we provide the result that a set of finite perimeter with constant horizontal normal is, in suitable coordinates, the upper graph of an entire Lipschitz function. 

\medskip

\begin{lemma}
\label{lemma:esisteG}
Let $E\subset \R^4$ be any an open $X_1$-invariant and $X_2$-monotone set.
Denote by $\bar\R$ the extended real line, i.e., $\overline\R:=\{-\infty\}\cup\R\cup\{+\infty\}.$
Then there exists a function $G:\R^2\to\overline \R$ such that
$$E=\{x\in\R^4\;:\;x_2>G(x_3,x_4)\}.$$
\end{lemma}
\proof
For each $x_3,x_4\in\R$, define
$G(x_3,x_4):=\inf\{ x_2:(0,x_2,x_3,x_4)\in E\}$. Here $\inf\{\emptyset\}=+\infty$.
Whenever such an infimum is finite, then it is not realized, since $E$ is open.
Since $E$ is  $X_2$-monotone and $(0,x_2,x_3,x_4)\cdot \exp(tX_2)=(0,x_2+t,x_3,x_4),$
we have that
$$E\cap(\{0\}\times\R\times\{x_3\}\times\{x_4\})= \{0\}\times(G(x_3,x_4),+\infty)\times\{x_3\}\times\{x_4\}.$$
For any $x\in\R^4$, since $E$ is $X_1$-invariant, we have that
$$x\in E\iff (0,x_2,x_3,x_4)\in E\iff x_2>G(x_3,x_4).$$

The upper semi-continuity of $G$ follows because $E$ is open.
\qed

\begin{lemma}
\label{lem:semispazio}
Let $\G$ be the Engel group in exponential coordinates of second kind with Lie algebra as in \eqref{VFE}.
Let $E\subset \G$ be an open $X_2$-calibrated set.  
Assume that there exists  $\tilde p\in E$
such that
$\tilde p+(0,\R ,0,0)\in E.$
Then 
$$\{x\in \R^4 : x_4>\tilde p_4\}\subseteq E.$$
\end{lemma}
\proof
Let 
$x\in \R^4$ with  $x_3\neq \tilde p_3$ and $ x_4>\tilde p_4$.
Set 
$s:=x_3-\tilde p_3$, which is nonzero, and $t:=\dfrac{s^2}{x_4-\tilde p_4}$, which is positive.
By the particular  assumption on $\tilde p$, 
we have
$$(\tilde p_1, x_2 - t,  \tilde p_3, \tilde p_4)\in E.$$
By $X_1$-invariance, 
$$(s/t, x_2 - t,  \tilde p_3, \tilde p_4)\in E.$$
By $X_2$-monotonicity, 
$$\left( \dfrac{s}{t}, x_2-t+t, \tilde p_3+\dfrac{s}{t}t,
\tilde p_4+\dfrac{s^2}{t^2}t\right)
\in E.$$
Explicitly, $$
\left( \dfrac{s}{t}, x_2, \tilde{p}_3+s,
\tilde{p}_4+ \dfrac{s^2}{t}\right)=
\left( \dfrac{s}{t}, x_2, \tilde{p}_3+x_3-\tilde p_3,
\tilde p_4+ {s^2}\dfrac{x_4-\tilde{p}_4}{s^2}\right)=
\left( \dfrac{s}{t}, x_2,  x_3 ,
    {x_4 } \right)
\in E.$$
By $X_1$-invariance, 
$$\left( x_1, x_2,  x_3 ,
    {x_4 } \right)
\in E.$$
\qed

\begin{observation}
\label{observation:semispazio}
The previous lemma is saying that the fuction $G$ describing $E$ has the property that the closure of the level set at $-\infty$ is a half-space orthogonal to $x_4$. With an analogous argument we can actually prove the stronger statement: if $(p^n_1, p^n_3, p^n_4) \to (p_1, p_3, p_4)$  as $n \to \infty$ and $G(p^n_1, p^n_3, p^n_4) \to -\infty$ then on the half-space $\{x_4 > p_4\}$ the function $G$ must take the value $-\infty$. We skip the proof of this statement, since it will easily follow from the properties of the set $C$ described in Example \ref{example3}.
\end{observation}

%

\begin{definition}[Partially Lipschitz map]
Let $G:\R^k\to\overline\R$, $v\in\R^k$, and $L>0$. We say that $G$ is {\em partially $L$-Lipschitz along $v$} if, for all $t>0$ and $x\in\R^k$, one has
$$G(x+tv) \leq Lt+G(x).$$
\end{definition}
Notice that in the above definition we only have a condition for positive $t$ and also for the difference $G(x+tv)-G(x)$, not for the absolute value. Example of partially Lipschitz maps are the monotone maps. Indeed, every nonincreasing function $G:\R\to\R$ is partially $1$-Lipschitz along $v$, for all $L>0$ and all $v>0$.

\begin{lemma}
\label{partiallylip} Let  $G:\R^2\to\overline\R$ be such that
the set $E=\{x\in\R^4\;:\;x_2>G(x_3,x_4)\}$ is  $X_2$-monotone. Then $G$ is partially $1$-Lipschitz along any vector $(a,a^2/2)$, with $a\in\R$.
\end{lemma}
\proof Fix $x_3,x_4\in\R$. Assume $G(x_3,x_4)\neq+\infty$, otherwise there is nothing to prove. 
Take $x_2>G(x_3,x_4)$. Thus $(a,x_2,x_3,x_4)\in E$. Since $E$ is $X_2$-monotone, we have that, for all $t>0$,
$$(  a,x_2+t,x_3+at,x_4+\dfrac{a^2}{2}t)\in E.$$
So $x_2+t>G(x_3+at,x_4+{a^2t}/{2})$,  for all $t>0$.
Letting $x_2\to G(x_3,x_4)$, we get
$$G(x_3,x_4)+t\geq G(x_3+at,x_4+{a^2t}/{2}),$$
which ends the proof. \qed

As a consequence we can get the following corollary, i.e. Theorem \ref{cor:37}:

\begin{cor}
\label{cor:lip}
There exist coordinates in which the set $E$ of constant normal $X_2$ can be expressed as upper-graph of a globally Lipschitz function of $\R^3$. 
\end{cor}

\begin{proof}
From the previous lemma, for any direction $v=(v_1, v_2)$ in $\R^2_{x_3, x_4}$ with $|v|=1$ and $v_2 >0$, the function $G$ is $\frac{(v_1)^2}{2 v_2}$ partially Lipschitz along $v$. So at every point $y$ on the graph of $G$ there is a cone-shaped domain $y + \{(x_1, x_2, x_3, x_4): x_4 \geq 0, \, x_2 > \frac{(x_3)^2}{2 x_4} \sqrt{x_3^2 + x_4^2}\}$ that is contained in the upper graph of $G$. 

Remark that the cone is independent of the point $y$ on the graph of $G$, it is just moved via (euclidean) translation. By suitably rotating coordinates, we can make therefore $E$ to be the upper-graph of a globally Lipschitz function: namely we have to choose a graphing direction that lies in the interior of the set $\{(x_1, x_2, x_3, x_4): x_4 \geq 0, \, x_2 > \frac{(x_3)^2}{2 x_4} \sqrt{x_3^2 + x_4^2}\}$.
\end{proof}

\subsection{The complete class of examples of sets with normal  $X_2$}
\label{classofex}
We present in this subsection a characterization (as well as some examples) of sets with normal $X_2$ in the model of the Engel group that we have used above. 

\medskip

\medskip

We recall a few facts on BV functions, with reference to \cite[pages 354-379]{Giaquinta-Modica-Soucek}.

Let $u$ be an $L^1_{loc}$ function on $\R^n$; the subgraph $SU$ of $u$, i.e., the set $\{(x,y) \in \R^n \times \R: y < u(x)\}$, is a set of locally finite (Euclidean) perimeter if and only if $u$ is $BV_{loc}$ (Thm. 1 page 371). 

Let $u:\R^n \to \R$ be $L^1_{loc}$. Define the approximate $\limsup$ and $\liminf$ at $x \in \R^n$ respectively as follows, where for any $t \in \R$ we use the notation $U_{t,u}:=\{x \in \R^n: u(x)>t\}$ and $L_{t,u}:=\{x \in \R^n: u(x)<t\}$:

$$u_+(x):= \sup\{t \in \R: \text{the $n$-dim. density of the set } L_{t,u} \text{ at } x \text{ is } 0\},$$

$$u_-(x):= \inf\{t \in \R: \text{the $n$-dim. density of the set } U_{t,u} \text{ at } x \text{ is } 0\}.$$

When $u_+(x)=u_-(x)$ we say that $x$ is a point of ``approximate continuity'' for $u$. The set $J$ of points where the strict inequality $u_+(x)>u_-(x)$ holds is the ``jump set'' of $u$. Then we have (see \cite[pages 355]{Giaquinta-Modica-Soucek}): the set $J$ is ${\Hc}^{n-1}$-measurable and countably ${\Hc}^{n-1}$-rectifiable.

The term ``jump set'' is justified by the result we are about to recall. Denote, for $x \in \R^n$ and $\nu_x \in S^{n-1}$, the half-space $\{y \in \R^n : \langle y-x, \nu \rangle >0\}$ by $E^+(x, \nu)$. Analogously denote the half space $\{y \in \R^n : \langle y-x, \nu \rangle <0\}$ by $E^-(x, \nu)$.

For ${\Hc}^{n-1}$-a.e. $x$ in $J$ there exists a (unique) $\nu_x \in S^{n-1}$ such that it holds:
$$ \text{aplim}_{y \to x, y \in E^+(x, \nu)} u(y) =u_+(x) \,\, \text{ and }  \,\, \text{aplim}_{y \to x, y \in E^-(x, \nu)} u(y) =u_-(x).$$
The notion of approximate limit here (see \cite[pages 210]{Giaquinta-Modica-Soucek}) is meant as follows: 

for all $\eps>0$ the set $\{y \in E^+(x, \nu) : |u(y)-u_+(x)|\geq \eps\}$ has $n$-dim. density $0$ at the point $x$. Analogously for all $\eps>0$ the set $\{y \in E^-(x, \nu) : |u(y)-u_-(x)|\geq \eps\}$ has $n$-dim. density $0$ at the point $x$.




Then we can improve our knowledge of $J$ with the following statement (\cite[pages 355]{Giaquinta-Modica-Soucek}): the set $J$ is ${\Hc}^{n-1}$-measurable and countably ${\Hc}^{n-1}$-rectifiable; moreover on $J$ we have that the approximate tangent space (in the sense of geometric measure theory) exists for ${\Hc}^{n-1}$-a.e. $x$ and is given by the orthogonal to $\nu_x$.

\medskip

Regarding the distributional derivative $Du$ of the $BV_{loc}$ function $u:\R^n \to \R$, we know that it is a locally finite measure (by definition). Setting $D^{(j)}u:= Du \res J$ and $\tilde D u:=Du -D^{(j)}u$ we are going to use the splitting $Du= \tilde D u + D^{(j)}u$. The measures $D^{(j)}u$ and $\tilde D u$ are mutually singular. There exists (see \cite{Ambook} or \cite{Giaquinta-Modica-Soucek}) a further splitting of $\tilde D u$ into an absolutely continuous (w.r.t. Lebesgue measure) part and a ``Cantor part'', but we are not going to need it for our purposes. The measure $D^{(j)}u$ is just $(u_+(x) - u_-(x)) ({\Hc}^1 \res J) \otimes \nu_x$.

\medskip

By recalling Theorems 2 and 3 on page 375 of \cite{Giaquinta-Modica-Soucek} we will now see how to express the distributional derivative $D\,\uno_{SU}$ of the characteristic function $\uno_{SU}$, for $u \in BV_{loc}$, in terms of $Du$.

We split $$D\,\uno_{SU} = D^{(j)}\,\uno_{SU} + D^{(cont)}\,\uno_{SU},$$
where $D^{(j)}\,\uno_{SU}:=D\,\uno_{SU} \res (J \times \R)$ and $D^{(cont)}\,\uno_{SU}:=D\,\uno_{SU} - D^{(j)}\,\uno_{SU} $.

Let $(x,y)$ denote the coordinates for $\R^n \times \R$. Then it holds, for $D^{(cont)}\,\uno_{SU}$:

$$\left(D_i^{(cont)}\,\uno_{SU}\right) (\phi(x,y)) = \int_{\R^n \setminus J} \phi(x,u_+(x)) D_i u \,\,\, \text{for any } \phi \in C^\infty_c(\R^n \times \R) \text{ and } i\in\{1, 2, ... n\},$$

\begin{equation}
\label{eq:cont}
 \left(D_{n+1}^{(cont)}\,\uno_{SU}\right) (\phi(x,y)) = -\int_{\R^n} \phi(x,u_+(x)) dx \,\,\, \text{for any } \phi \in C^\infty_c(\R^n \times \R).
\end{equation}

Regarding the jump part we have that in $\R^n \times \R$
\begin{equation}
 \label{eq:jump}
D^{(j)}\,\uno_{SU} = \left({\Hc}^n \res V\right) \otimes (\nu_x, 0),
\end{equation}
where $V=\{(x,y) \in \R^n \times \R: x \in J, u_-(x) < y < u_+(x)\}$ and the vector $\nu_x$ is the normal to $J$ in $\R^n$.

\medskip

\medskip

We are now ready to prove:

\begin{prop}
\label{prop:characterize}
In our model of the Engel group a set has finite perimeter and constant horizontal normal $X_2$ if and only if it is of the form  $\{x_2 > G(x_3, x_4)\}$ for an upper semi-continuous function $G:\R^2 \to \overline{\R}$ with the following properties:
\begin{description}
 \item[(i)] the closure of $\{(x_3, x_4) : G(x_3, x_4) = -\infty\}$ is a half-space of the form $\{(x_3, x_4) : x_4 \geq b\}$ for some $b \in \overline{\R}$;

 \item[(ii)] the restriction of $G$ to the open set $\mathcal{G}:= \R^2 \setminus \overline{\{(x_3, x_4) : x_4 \geq b\}} \setminus \{(x_3, x_4) : G(x_3, x_4) = +\infty\}$ is $\text{BV}_{loc}(\mathcal{G})$;

 \item[(iii)] $G$ satisfies the following partial differential inequality\footnote{This is the distributional analogue of the inequality $( \partial_3G)^2+2\partial_4 G\leq 0$ in the case that $G$ is a smooth function. 
Indeed, assuming $( \partial_3G)^2+2\partial_4 G\leq 0$, for any $h\in C^\infty_c(\mathcal{G})$ such that $h\geq 0$ and $\int h =1$, we have $$\int ( \partial_3G)^2 h+2 \int \partial_4 G \,  h \leq 0.$$ Jensen's inequality applied with respect to the measure of unit mass $h\, d{\Lc}^2$ yields $$\left(\int \partial_3G \, h\right)^2 \leq \int ( \partial_3G)^2 h.$$
On the other hand, by assuming (\ref{PDI}) and using it on a sequence of test functions $h_n$ having unit integral and converging to the Dirac delta at a point, we pointwise obtain the inquality $( \partial_3G)^2+2\partial_4 G\leq 0$.
 } on $\mathcal{G}$:
for all $h\in C^\infty_c(\mathcal{G})$ such that $h\geq 0$,  it holds
\begin{equation}\label{PDI}
(\langle \partial_3G, h\rangle)^2+2\langle\partial_4 G, h\rangle \langle {\Lc}^2, h \rangle \leq 0 .
\end{equation}
Here  ${\Lc}^2$ denotes the Lebesgue measure on $\R^2$ and  $\langle \,,\,\rangle$ denotes the pairing of distributions and smooth test functions.

\end{description}
\end{prop}

\begin{observation}
The inequality (\ref{PDI}) can be equivalently expressed by requiring that, for any $h\in C^\infty_c(\mathcal{G})$ such that $h\geq 0$ and $\int h =1$, it holds

\begin{equation}\label{PDI2}
\left(\int_{\mathcal{G}} G \frac{\partial h}{\partial x_3} \right)^2 \leq 2\int_{\mathcal{G}} G \frac{\partial h}{\partial x_4} .
\end{equation}
\end{observation}

\medskip

We first show the following

\begin{lemma}
\label{lem:split}
Let  $G: \mathcal{G} \subset \R^2 \to\R$ be as in Proposition \ref{prop:characterize} and be $J \subset \R^2$ its jump part. We take $G$ to be a function of the variables $x_3$ and $x_4$ and we will denote by $\partial_3$ (resp. $\partial_4$) the partial derivative, which is a Radon measure, with respect to the variable $x_3$ (resp. $x_4$). 

The partial differential inequality (\ref{PDI}) splits into (and actually is equivalent to)

\begin{equation}
\label{splittedeq}
(\langle \tilde{\partial}_3G, h\rangle)^2+2\langle\tilde{\partial}_4 G, h\rangle \langle {\Lc}^2, h \rangle \leq 0\,\, , \,\,\,\,\,\, \partial^{(j)}_3G= 0,
\end{equation} 
where we are using the notation $\tilde{\partial}$ and $\partial^{(j)}$ introduced before and $h$ is any non-negative test function. The second in (\ref{splittedeq}) is equivalent to saying that $J$ has normal $\nu_x$ that is parallel to the $x_4$-direction for ${\Hc}^1$-a.e. $x \in J$.
\end{lemma}

\begin{proof}[Proof of Lemma \ref{lem:split}]

 We shall prove that (\ref{PDI}) yields the two inequalities in (\ref{splittedeq}).

By definition of ${\Hc}^{n-1}$-rectifiable we have $J=\cup_{i=1}^\infty f_i(K_i)$, where the $K_i$'s are compact sets in $\R$ and $f_i$'s are Lipschitz functions from $\R$ to $\mathcal{G}$. We can assume the union $\cup_{i=1}^\infty f_i(K_i)$ to be disjoint. Fix any $\eps >0$: for each $N$ we can find an open neighbourhood $A_{N,\eps}$ of the compact set $\cup_{i=1}^N f_i(K_i)$ such that ${\Lc}^2(A_{N,\eps}) \leq \eps$. This is achieved by taking neighbourhoods of each $f_i(K_i)$ having measure at most $\frac{\eps}{2^i}$ and taking their union from $i=1$ to $i=N$. The fact that we can find an open neighbourhood of $f_i(K_i)$ having arbitrarily small area is a consequence of the fact that $f_i(K_i)$ has finite ${\Hc}^1$-measure. 

Choose now, for $N$ and $\eps$ fixed, a smooth bump function $\psi_{N,\eps}$ that is identically $1$ on the compact set $\cup_{i=1}^\infty f_i(K_i)$, identically $0$ outside of $A_{N,\eps}$ and takes values between $0$ and $1$. 

\medskip

Choose any $h \in C^\infty_c(\mathcal{G})$. The partial differential inequality (\ref{PDI}) used on the function $h_{N,\eps}:=h \psi_{N,\eps}$ reads

\begin{equation}
\label{reads}
 (\langle \tilde{\partial}_3 G, h_{N,\eps}\rangle+ \langle  \partial^{(j)}_3 G, h_{N,\eps}\rangle)^2+2\langle \tilde{\partial}_4 G, h_{N,\eps}\rangle \langle {\Lc}^2, h_{N,\eps} \rangle + 2\langle  \partial^{(j)}_4 G, h_{N,\eps}\rangle \langle {\Lc}^2, h_{N,\eps} \rangle\leq 0 .
\end{equation}
Keeping $N$ fixed and letting $\eps \to 0$ we get that 
$$ \tilde{\partial}_3 G \left(A_{N,\eps}\right) \to \tilde{\partial}_3 G \left(\cup_{i=1}^N f_i(K_i)\right) =0, $$
where the convergence holds since $\cup_{i=1}^N f_i(K_i)= \cap_{\eps>0}A_{N,\eps}$, $\tilde{\partial}_3 G$ is a Radon measure and the sets $A_{N,\eps}$ are bounded.
This, together with an analogous convergence for $\tilde{\partial}_4 G$ and ${\Lc}^2$, gives that for $N$ fixed and $\eps \to 0$:
$$\langle \tilde{\partial}_3 G, h_{N,\eps}\rangle \to 0 \, , \,\,\, \langle \tilde{\partial}_4 G, h_{N,\eps}\rangle \to 0, \,\,\,\, \langle {\Lc}^2, h_{N,\eps} \rangle \to 0.$$
Let us now look at the remaining terms in (\ref{reads}), namely those involving the ``jump parts''. Denote by $\nu_{3,N}$ the measure $ \partial_3^{(j)}G \res \left(\cup_{i=1}^N f_i(K_i)\right)$. In the same fashion let $\nu_{4,N} := \partial_4^{(j)}u \res \left(\cup_{i=1}^N f_i(K_i)\right)$. Remark that $\nu_{3,N} \rightharpoonup \partial_3^{(j)}G$ and $\nu_{4,N} \rightharpoonup \partial_4^{(j)}G$ as $N \to \infty$.

For a fixed $N$ we get moreover (recall that $\psi_{N,\eps}=1$ on $\cup_{i=1}^N f_i(K_i)$) that, as $\eps \to 0$:

 $$\langle  \partial^{(j)}_3 G, h_{N,\eps}\rangle \to \langle  \nu_{3,N}, h\rangle \,, \,\,\, \langle  \partial^{(j)}_4 G, h_{N,\eps}\rangle \to \langle  \nu_{4,N}, h\rangle.$$
So we can send (\ref{reads}) to the limit for $\eps \to 0$ and get
\begin{equation}
\label{reads2}
 (\langle  \nu_{3,N}, h\rangle)^2 \leq 0 ,
\end{equation}
which holds for every $h \geq 0$. Using the convergence of measures $\nu_{3,N} \rightharpoonup \partial_3^{(j)}G$ as $N \to \infty$ we obtain that $\partial_3^{(j)}G=0$, as in (\ref{splittedeq}). This equivalently means that $J$ has a normal $\nu$ always parallel to the $x_4$ direction.

\medskip

In order to get the first inequality in (\ref{splittedeq}) we can use an analogous argument, this time using $1-\psi_{N,\eps}$ instead of $\psi_{N,\eps}$.
\end{proof}

\begin{observation}
 the condition on the shape of $J$ is actually equivalent to ${\Hc}^1\left(J \setminus \cup_{i=1}^\infty B_i\right)=0$, where each $B_i$ is a Borel subset of a line parallel to $x_3$.
\end{observation}

\begin{observation}
 It is easily seen that, for $G \in BV_{loc}$, equations (\ref{splittedeq}) are actually equivalent to (\ref{PDI}).
\end{observation}

\begin{proof}[Proof of Proposition \ref{prop:characterize}]

As we saw in Lemma \ref{lemma:esisteG}, every set having locally finite horizontal perimeter and constant horizontal normal equal to $X_2$ is the uppergraph of a function $G:\R^2 \to \overline{\R}$ of the variables $(x_3, x_4)$. Such $G$ is upper semi-continuous and by Lemma \ref{lem:semispazio} the closure of the level set at $-\infty$ is a (closed) half-space in the direction $x_4$. Such a $G$ will then be $L^1_{loc}$ on the open set $\mathcal{G}$ (see observation \ref{observation:semispazio}). 

We have moreover seen that $E$ has Lipschitz boundary (in the Euclidean sense) when we choose suitable coordinates (Lemma \ref{partiallylip}). This makes it a set of locally finite Euclidean perimeter; thus, since being of locally finite Euclidean perimeter is a notion which is independent of coordinates, going back to the original coordinates the function $G$ must be $BV_{loc}$ on $\mathcal{G}$.

\medskip

We thus need to prove that, for $G$ as in assumptions (i) and (ii), the set $E:=\{x\in\R^4\;:\;x_2>G(x_3,x_4)\}$ is $X_2$-monotone if and only if $G$ satisfies (\ref{PDI}).

\medskip

Regarding $X_2$-monotonicity, we split the derivatives $\partial_{x_j} \uno_E$ in the ``approximately continuous part'' and the ``jump part''.

\medskip

We can compute, on the ``approximately continuous part'' $(\R^3_{x_1, x_3, x_4} \setminus (\R_{x_1} \times J)) \times \R_{x_2}$, the horizontal normal to $\uno_E$ as follows: for any non-negative $h \in C^\infty_c(\R^4)$ it holds (from (\ref{eq:cont}))

\begin{equation}
 \label{eq:contpiece}
\left[\left(\partial_2 +x_1 \partial_3 + \dfrac{x^2_1}{2} \partial_4 \right)^{\text{cont}}\uno_E \right]\left(h\right) = $$ $$=
\int_{\R^3_{x_1, x_3, x_4} \setminus (\R_{x_1} \times J)} h\left(x_1, G(x_3,x_4), x_3, x_4\right) \left[1-x_1(\partial_3G)(x_3,x_4)-\dfrac{x^2_1}{2}(\partial_4 G)(x_3,x_4)\right].
\end{equation}

We now consider $J \times \R$ and (recall (\ref{eq:jump})) here we have $\partial_2 \uno_{E}=0$. Let further $(a,b)$ be the vector $(\partial_3 \uno_{E}, \partial_4 \uno_{E})$. Then 

\begin{equation}
 \label{eq:jumppiece}
\left(\partial_2 +x_1 \partial_3 + \dfrac{x^2_1}{2} \partial_4 \right)^{\text{(j)}}\uno_E  = x_1 (u^+ - u^-) a ({\Hc}^1\res J) + \frac{x_1^2}{2}(u^+ - u^-) b ({\Hc}^1\res J).
\end{equation}

Altogether, summing the two expressions in (\ref{eq:contpiece}) and (\ref{eq:jumppiece}), we get the expression for the $X_2$-derivative of $\uno_E$. The condition of $X_2$-monotonicity, i.e. that $\left[\left(\partial_2 +x_1 \partial_3 + \dfrac{x^2_1}{2} \partial_4 \right) \uno_E\right](h)$ be positive for any $h\geq 0$ and for any $x_1$ is fulfilled if and only if\footnote{Indeed the two measures in (\ref{eq:contpiece}) and (\ref{eq:jumppiece}) are mutually singular and the inequality $\left[\left(\partial_2 +x_1 \partial_3 + \dfrac{x^2_1}{2} \partial_4 \right) \uno_E\right](h)$ splits in the two corresponding inequalities for the two measures $\left(\partial_2 +x_1 \partial_3 + \dfrac{x^2_1}{2} \partial_4 \right)^{\text{cont}}\uno_E$ and $\left(\partial_2 +x_1 \partial_3 + \dfrac{x^2_1}{2} \partial_4 \right)^{\text{(j)}}\uno_E$. This is proved using bump functions as done in the proof of Lemma \ref{lem:split}.} for any $h \geq 0$, the polinomial in $x_1$ 

$$\int h \, d{\Lc}^2 -x_1\langle(\partial_3G)(x_3,x_4), h \rangle -\dfrac{x^2_1}{2}\langle(\partial_4 G)(x_3,x_4),h \rangle$$
is always positive and 
$$\left(x_1 a + \dfrac{x^2_1}{2} b \right) \geq 0.$$

The first is in turn equivalent, since such a polynomial has value $1$ for $x_1=0$, to the discriminant $(\partial_3G(h))^2+2\partial_4 G(h) \langle {\Lc}^2, h \rangle$ being nonpositive.

The second is satisfied if and only if $a=0, b\geq 0$. The vector $(a,b)$ is, on the other hand, the normal $\nu$ to the jump set $J \subset \R^2$ of $G$: so the $X_2$-monotonicity is equivalent to $J$ being a countably ${\Hc}^{1}$-rectifiable set with constant normal in the direction $x_4$, as in the assumptions.
\end{proof}

\medskip

We give now some explicit examples of sets having locally finite horizontal perimeter and constant horizontal normal equal to $X_2$ in our model of the Engel group. The first one is a generalization of Example \ref{example}.


\begin{example}\label{example2}
Let $g:\R\to\R$ be a non-increasing and upper semi-continuous function. 
Take  $K \in (0, \infty)$ and  a non-decreasing function $f:\R\to \R$ Lipschitz continuous with Lipschitz constant $\leq \frac{2}{K^2}$. The set $$E:=\{x\in\R^4 : x_2 > f(Kx_3 - x_4) + g(x_4)\}$$ is $X_2$-calibrated. We can easily see this fact as a consequence of Proposition \ref{prop:characterize} by computing
$$\left((\partial_3 G)(x_3, x_4)\right)^2 = K^2 (f'(Kx_3-x_4))^2, (\partial_4 G)(x_3, x_4) = -f'(Kx_3-x_4) + \partial_4 g (x_4)$$ 
so that $$\left((\partial_3 G)(x_3, x_4)\right)^2 + 2 (\partial_4 G)(x_3, x_4) = (K^2 f'(Kx_3-x_4) -2)f'(Kx_3-x_4) +2 \partial_4 g (x_4)$$
$$\leq  2 \partial_4 g (x_4) \leq 0$$
by the condition on the Lipschitz constant of $f$ and by the monotonicity of $g$. 
\end{example}

\begin{example}\label{example3}
The set $$C:=\left\{x\in\R^4 : x_2 >0, x_4 >0, x_2 > \frac{x_3^2}{2 x_4}\right\}$$ is $X_2$-calibrated. In this case we have $G=+\infty$ for $x_4 \leq 0$ and $G=\frac{x_3^2}{2 x_4}$ for $x_4 >0$.

Again, making use of Proposition \ref{prop:characterize}, we can compute, for $x_4 >0$:
$$\left((\partial_3 G)(x_3, x_4)\right)^2 + 2 (\partial_4 G)(x_3, x_4) = \frac{x_3^2}{x_4^2} - \frac{x_3^2}{x_4^2} =0.$$

We remark here that we get $0$ because $C$ is a sort of ``extreme case'', in the sense that, taken any $X_2$-calibrated set $E$, if $p \in \partial E$ then $p+C$ must lie in the interior of $E$. This fact will be discussed in detail and play an important role in a subsequent work.

Let us prove the previous claim. Assume that $E$ is $X_1$-invariant and $X_2$-monotone and let $p=(p_1,p_2, p_3, p_4) \in E$. Then the whole line $\ell=\{(p_1+a,p_2, p_3, p_4)\,:\;a \in \R\}$, belongs to $E$. Now the $X_2$-monotonicity means that we can flow from any point in $\ell$ for positive times $t$ along $X_2$ and we remain in $E$. Writing this down explicitly we get
$$\left(p_1+a,p_2, p_3, p_4\right) + \left(0, t, (p_1+a)t, \frac{(p_1+a)^2}{2} t \right) \in E \,\,\, \text{ for any } t \geq 0,  a \in \R.$$
By using $X_1$-invariance again we get that
$$\left(p_1,p_2, p_3, p_4\right) + \left(0, t, (p_1+a)t, \frac{(p_1+a)^2}{2} t \right) \in E \,\,\, \text{ for any } t \geq 0,  a \in \R.$$
The points $\left(t, (p_1+a)t, \frac{(p_1+a)^2}{2} t \right)$ with $t \geq 0$ and $a \in \R$ describe the surface $\{2x z = y^2: x >0, z>0\}$ in $\R^3$. This means that whenever $E$ contains $p$ then it must contain the surface $p + \{(x_1, x_2, x_3, x_4): x_2 = \frac{x_3^2}{2 x_4}, x_2 >0, x_4>0\}$. But recalling that $E$ is an upper-graph in the direction of the $x_2$-coordinate we get that $E$ contains the whole $p + \{(x_1, x_2, x_3, x_4): x_2 \geq \frac{x_3^2}{2 x_4}, x_2 >0, x_4>0\}$, which is exactly $p + C$.

\end{example}

\section{\textbf{Sets with normal $X_2$ as intrinsic graphs}}
\label{intrinsic}

In this section we look at the expression of $E$ (a set with constant normal $X_2$) as subgraph of a function when we use as ``graphing direction'' the flow lines of an horizontal vector field. The most natural choice would be to use the flow of $X_2$ as ``graphing direction'', as we are about to explain in the next subsection.

\subsection{Intrinsic graphs in the direction of the normal}
Let $W\subseteq\R^4$ be the set of points with second component equal to zero, 
$$W:=\{p\in \R^4: p_2=0\}.$$
One can show that $W$ is a subgroup of $\R^4$ with respect to the Engel structure.
Indeed, to see this, it is enough to observe the following two facts.
First, the vector space spanned by the vector fields $X_1,X_3,X_4$ form a Lie sub-algebra.
The second fact to notice is that the span of such vectors is tangent to $W$.
Thus $W$ is a subgroup whose Lie algebra has basis  $X_1,X_3,X_4$.
From the algebraic viewpoint, the subgroup $W$ is a complementary subgroup of the one-parameter subgroup tangent to the vector field $X_2$.
From the geometric viewpoint, for each $p\in\R^4$, the $3$-dimensional plane $W$ intersects the line $t\mapsto p\exp(t X_2)$ in one and only one point.
Indeed, by \eqref{Flow X2} the second coordinate of  $p\exp(t X_2)$  is 
$p_2+t$, which is zero when (and only when) $t=-p_2$.
We conclude that the space $\R^4$ can be parametrized by the following map
$$\Psi: W\times \R \to\R^4$$
$$ (p,t)\mapsto p\exp(t X_2).$$
Assume now that $E\subseteq\R^4$ is a $X_2$-calibrated set.
We plan to write $E$ as an upper graph of a function. By the definition of $X_2$-calibration there exists a map $p\mapsto T(p)$ from $\R^4$ to $\R$ such that
$$\{t : p\exp(t X_2))\in E\}=(T(p),+\infty).$$
Restricting such a map $T$ to $W$. We get that 
$$E= \{ \Psi(p,t) : p\in W, t> T(p)\}$$
$$\quad \quad\quad=\{\Psi(p,T(p)+t) : p\in W, t> 0\}.$$
$$\quad\quad \quad \quad \quad \quad \quad \quad \quad =\{(p_1, t, p_3+p_1 t, p_4 + p_1^2t/2) : p\in W, t> T(p)\}.$$

Let us study the map $p\mapsto T(p)$ from $W$ to $\R$, in the examples  Example \ref{example} where
 $E:=\{x\in\R^4 : x_2 + g(x_4)> 0\},$
with $g$ non-decreasing and upper semicontinuous. 
The value $T(p)$ is the lower value $T$ such  that
$$T+g(p_4+p_1^2 T)\geq 0,$$
since on $W$ we have $p_2=0$.
Restrict the map $T$ to $W\cap \{p_1=0\}$, so
$$T+g(p_4)=0.$$
In conclusion, $T:W\to \R$ is  as much non-regular as $g$ is. In particular, there are examples of non-continuous function $T$. We can summarize the last discussion in the following fact. 
\begin{theorem}
\label{thm:noncont}
There are examples of $X_2$-calibrated sets $E$ with the property that, when they are  written as upper graphs in the  direction of $X_2$, the function 
for which they are upper graphs is  not continuous.  
\end{theorem}

\subsection{Graphs in other horizontal directions}

For a set $E$ with constant normal $X_2$, the previous example has shown that there can be a lack of continuity for the intrinsic graph representing the boundary of $E$ when we choose the flow lines of $X_2$ as graphing directions.

We might however still look at what happens when the boundary of $E$ is represented as an intrinsic graph using different horizontal flow lines as graphing direction: namely let us observe the flow lines of $aX_1 + X_2$ for $a>0$.

First of all we need to write down, analogously to what was done in (\ref{ODE flow 2}), the flow of $aX_1 + X_2$ in the model of the Engel group considered so far. What we get is that the flow line starting at $(p_1, p_2, p_3, p_4)$ is

\begin{equation}
\label{flowaX1}
\left( p_1+at, p_2 + t, p_3 + p_1 t + \frac{a}{2} t^2 , p_4 + \frac{p_1^2}{2}t + \frac{a p_1}{2} t^2 + \frac{a^2}{6} t^3 \right).
\end{equation}

We are going to show now that this intrinsic graph might fail to be Lipschitz. To see this, we will consider the set in Example \ref{example3}.

\begin{theorem}
\label{theorem:notlipshitz}
When the set $C:=\left\{x\in\R^4 : x_2 >0, x_4 >0, x_2 > \frac{x_3^2}{2 x_4}\right\}$ is represented as intrinsic upper graph in any horizontal direction $aX_1 + X_2$, the function for which it is upper graph is not Lipschitz.
\end{theorem}

\begin{proof}[Proof of Theorem \ref{theorem:notlipshitz}]
The intrinsic function $T(p)$ yielding the upper graph is a function $T:W\to \R$, where $W:=\{p\in \R^4: p_2=0\}$ as before. The value $T(p)$ is the infimum of the times $t$ for which the flow line starting at $p$ is inside $E$.

Let us restrict the attention to points in $W$ with $p_1=p_3=0$. The flow lines are then $\left( at, t,  \frac{a}{2} t^2 , p_4 + \frac{a^2}{6} t^3 \right)$ and we must see when the flow line enters the set $C$. The value $T(p)$ for $p_4 \geq 0$ is clearly $0$, while for $p_4<0$ it is the solution of the following equation in $t$
$$2 t \left(p_4 + \frac{a^2}{6} t^3\right) =  \frac{a^2}{4} t^4\, ,$$
with the constraint that $p_4 + \frac{a^2}{6} t^3>0$.

Solving this equation we get $\frac{a^2}{12}t^3 = 2|p_4|$, i.e. $t= \sqrt[3]{\frac{24|p_4|}{a^2}}$, which fulfils the constraint $p_4 + \frac{a^2}{6} t^3>0$. So we have that, restricting to $p_1=p_3=0$ in $W$, the function for which $C$ is upper-graph is $\sqrt[3]{\frac{24|p_4|}{a^2}}$ for $p_4<0$ and identically $0$ for $p_4 \geq 0$, so it is not Lipschitz continuous.
\end{proof}

It is therefore necessary to use non-horizontal directions as ``graphing direction'' (as done in corollary \ref{cor:lip}) in order to see the Lipschitz continuity of the function describing the boundary of a set with constant horizontal normal. 

The previous proof leaves however still open the possibility for the intrinsic graph in the direction $aX_1 + X_2$ to be Lipschitz with respect to the intrinsic Carnot-Caratheodory distance on $W$ (in particular H\"older continuous for the Euclidean distance), as we are about to discuss.

\begin{lem}
\label{lem:holder}
When a $X_2$-calibrated set $E$ is represented as intrinsic upper graph in any horizontal direction $aX_1 + X_2$, with $a \neq 0$, the function for which it is upper graph is H\"older-continuous.
\end{lem}

\begin{proof}[Proof of Lemma \ref{lem:holder}]
We know that, whenever $p \in \partial E$, then the set $p+C$ is contained in the interior of $E$, where $C$ is the set in Example \ref{example3}. In order to prove the theorem it therefore suffices to show that $C$ itself is, in the direction $aX_1 + X_2$, the upper-graph of a function that is H\"older-continuous at the origin. Analogously to what was done in proof of Theorem \ref{theorem:notlipshitz}, the intrinsic graph is the lowest value of $t$ such that
$$\frac{a^2}{12}t^4-p_3 a t^2 + 2 p_4 t - p_3^2 \geq 0\, ,$$
for $p_4<0$ with the constraint that $p_4 + \frac{a^2}{6} t^3>0$.





At $t=0$ the polynomial is clearly negative. Let us substitute the value $t=K \max\{\sqrt{|p_3|}, \sqrt[3]{|p_4|}\}$, for some constant $K>0$. 

Then we get, in the case $\sqrt{|p_3|}> \sqrt[3]{|p_4|}$,
$$K^4 \frac{a^2}{12}|p_3|^2-K^2 p_3 a |p_3| - 2 K |p_4| |p_3|^{\frac{1}{2}} - p_3^2 \geq  K^4 \frac{a^2}{12}|p_3|^2-\left[K^2  a |p_3|^2 + 2 K |p_3|^{\frac{3}{2}} |p_3|^{\frac{1}{2}} + |p_3|^2\right], $$
and this expression is positive if $K$ is chosen suitably large.

In the remaining case $\sqrt{|p_3|} \leq \sqrt[3]{|p_4|}$ we obtain 
$$K^4 \frac{a^2}{12}|p_4|^{\frac{4}{3}}-K^2 p_3 a |p_4|^{\frac{2}{3}} - 2 K |p_4| |p_4|^{\frac{1}{3}} - p_3^2 \geq  K^4 \frac{a^2}{12}|p_4|^{\frac{4}{3}}-\left[K^2  a |p_4|^{\frac{4}{3}} + 2 K |p_4|^{\frac{4}{3}}  + |p_4|^{\frac{4}{3}}\right], $$
and again this is positive if $K$ is chosen suitably large.

Altogether we have shown that for any $p \in W$ there exists a value $T(p)$ yielding the function describing $C$ as intrinsic upper graph in the direction $aX_1 + X_2$ and $T(p)$ is below $K \max\{\sqrt{|p_3|}, \sqrt[3]{|p_4|}\}$. In other words $T(p)$ at the origin has at least the H\"older regularity of the function $K \max\{\sqrt{|p_3|}, \sqrt[3]{|p_4|}\}$. 
\end{proof}

As a corollary of the previous proof we can state

\begin{theorem}
\label{theorem:intrinsiclipshitz}
When a $X_2$-calibrated set $E$ is represented as intrinsic upper graph in any horizontal direction $aX_1 + X_2$, with $a \neq 0$, the function for which it is upper graph is intrinsically Lipschitz-continuous.
\end{theorem}

\begin{proof}[Proof of Theorem \ref{theorem:intrinsiclipshitz}]
 In order to prove the theorem we must prove that whenever $p \in \partial E$ then there exists an intrinsic cone at $p$ whose interior part is all contained in $E$. In the proof of the previous lemma we have seen that the uppergraph of $K \max\{\sqrt{|p_3|}, \sqrt[3]{|p_4|}\}$ is all contained in $C$, and we know that $p+C$ is all contained in $E$.

Now the upper  graph of $K \max\{\sqrt{|p_3|}, \sqrt[3]{|p_4|}\}$ contains an intrinsic cone with respect to the Carnot-Caratheodory distance on $W$, thanks to the ball-box Theorem (see \cite{LeDonne-lectures}).
\end{proof}

\section{\textbf{Regularity in filiform groups}}\label{filif-sec}

A stratified group $\G$ is said to be a {\em filiform group} if the strata $V_j$ of the stratification
$${\rm Lie}(\G)=V_1\oplus \ldots \oplus V_s$$
 of the  Lie algebra Lie$(\G)$ are such that $\dim V_1=2$ and $\dim V_j=1$, for $j=2,\ldots , s$. Here $s$ is the step of the group.
 
 One can easily show that there exists a basis $X_0, \ldots X_s$ of  Lie$(\G)$ by vectors with the following property:
 $X_0, X_1\in V_1$, $X_j\in V_j$, for $j=2,\ldots , s$, and 
 \begin{equation}
 \label{filif}
 [X_0,X_{j-1}]=X_j,\qquad\text{ for }j=2,\ldots , s.
 \end{equation} In general there might be other non-null brackets of   elements of this basis.
In \cite{Vergne}, Vergne gave a classification of all stratified groups.  
In fact, in \cite[Corollary 1, page 93]{Vergne}, Vergne showed that, in the case the step $s$ is even (so the dimension of the group is odd), then there is only one stratified group of step $s$ and  it admits a basis for which the   brackets   are all null, except those in \eqref{filif}. In case $s$ is odd (and the dimension is even), then there are only two different filiform groups: one where, a part from \eqref{filif},  all other brackets are null and a second one where the only other non-null bracket relation is
$$[X_l, X_{s-l}]=(-1)^lX_s, \qquad \text{ for } l=1, \ldots, s-1.$$
We refer to this two groups as the
filiform group  of the {\em first kind} and  the
filiform group   of the {\em second kind}, respectively.

We shall show how the argument for proving the regularity
of constant-normal sets in the Engel group
can be extended to any filiform group of the first kind.
\begin{theorem}\label{filif-thm}
Let $\G$ be any filiform group of the first kind. Let $E\subset \G$ be a set with horizontal constant normal. Then $E$ is a Lipschitz domain.
\end{theorem}
\proof
Let   $X_0, \ldots X_s$ be a basis of  Lie$(\G)$ satisfying  \eqref{filif}. As for the Engel group, we can assume that either $X_0\uno_E=0$ or  $X_1\uno_E=0$.
Consider first the case  $X_0\uno_E=0$.
By Proposition \ref{AKL prop}, 
the vector ${\rm Ad}_{\exp(tX_0)}X_1$ is a monotone direction.
Explicitly, by  \eqref{filif},  we have
$${\rm Ad}_{\exp(tX_0)}X_1=e^{{\rm ad} (tX_0)}X_1=X_1+tX_2+\dfrac{t^2}{2}X_3+\ldots +\dfrac{t^{s-1}}{s-1}X_s.$$
Pick $s$ distinct numbers $t_1,\ldots, t_s$. Consider the vectors $$Y_j={\rm Ad}_{\exp(t_jX_0)}X_1, \quad \text{ for } j=1,\ldots, s.$$
 We claim that the vectors $Y_j$ are linearly independent. Indeed, it is enough to show that the matrix
 \begin{equation*}
 \begin{pmatrix}
	1&t_1&\dfrac{t_1^2}{2}&\ldots&\dfrac{t_1^{s-1}}{s-1} \\
\vdots &\vdots&\vdots&\ddots&\vdots\\
	1&t_s&\dfrac{t_s^2}{2}&\ldots&\dfrac{t_s^{s-1}}{s-1}
     \end{pmatrix}.\end{equation*}
has full rank.
Equivalently we need 
 \begin{equation*}
\det \begin{pmatrix}
	1&t_1&t_1^2&\ldots& {t_1^{s-1}}{ } \\
\vdots &\vdots&\vdots&\ddots&\vdots\\
	1&t_s&t_s^2&\ldots& {t_s^{s-1}}{ }
     \end{pmatrix}\neq 0.\end{equation*}
We observe that we are considering a  Vandermonde Matrix. Hence
such a determinant is $\Pi _{1\leq i<j<s}(t_i-t_j)$, which is nonzero, since the $t_j$'s have been chosen to be distinct.
Since we found a basis of monotone directions, as for the Engel group, we conclude that the set $E$ is (equivalent) to a Lipschitz domain.

Let us consider now the case  $X_1\uno_E=0$.
Applying Proposition \ref{AKL prop},   we get that the vector field 
$$ {\rm Ad}_{\exp(tX_1)}X_0 =X_0-tX_2$$
is a monotone direction, for all $t\in\R$. Thus $X_2\uno_E=0$.
Iterating the use of  Proposition \ref{AKL prop} and using   \eqref{filif}, we get that all vectors $X_2, \ldots, X_s$ are invariant directions. Hence $E$ is  half-space.
\qed

   \bibliography{general_bibliography}
\bibliographystyle{amsalpha}

\end{document}